\documentclass[12pt, reqno]{amsart}
\usepackage{amsmath, amsthm, amscd, amsfonts, amssymb, graphicx, color}
\usepackage[bookmarksnumbered, colorlinks, plainpages]{hyperref}
\hypersetup{colorlinks=true,linkcolor=red, anchorcolor=green, citecolor=cyan, urlcolor=red, filecolor=magenta, pdftoolbar=true}

\usepackage{tikz}
\newtheorem{theorem}{Theorem}[section]
\newtheorem{lemma}[theorem]{Lemma}
\newtheorem{proposition}[theorem]{Proposition}
\newtheorem{cor}[theorem]{Corollary}
\theoremstyle{definition}
\newtheorem{defn}[theorem]{Definition}

\theoremstyle{remark}
\newtheorem{remark}[theorem]{\bf{Remark}}
\numberwithin{equation}{section}
\begin{document}

	\title [Numerical radius of sectorial matrices] 
	{   {Numerical radius inequalities of sectorial matrices} }

	\author[P. Bhunia, K. Paul, A. Sen]{Pintu Bhunia,  Kallol Paul, Anirban Sen}
	
	\address{(Bhunia) Department of Mathematics, Jadavpur University, Kolkata 700 032, West  Bengal, India}
	\email{pintubhunia5206@gmail.com}
	\email{pbhunia.math.rs@jadavpuruniversity.in}
	
	\address{(Paul) Department of Mathematics, Jadavpur University, Kolkata 700 032, West      Bengal, India}
	\email{kalloldada@gmail.com }
	\email{kallol.paul@jadavpuruniversity.in}
	
	\address{(Sen) Department of Mathematics, Jadavpur University, Kolkata 700 032, West  Bengal, India}
	\email{anirbansenfulia@gmail.com}

	%\thanks will become a 1st page footnote.
	\thanks{ Pintu Bhunia would like to thank UGC, Govt. of India for the financial support in the form of Senior Research Fellowship under the mentorship of Prof. Kallol Paul. Anirban Sen would like to thank CSIR, Govt. of India for the financial support in the form of Junior Research Fellowship under the mentorship of Prof. Kallol Paul. }
	\thanks{}
	\thanks{}
	
	%    General info
	
	
	%\date{}
	\maketitle

	\begin{abstract}
		We obtain  several upper and lower bounds for the numerical radius of sectorial matrices. We also develop several numerical radius inequalities of the sum, product and commutator of sectorial matrices. The inequalities obtained here are sharper than the existing related inequalities  for general matrices. Among many other results  we prove  that
		if  $A$ is an $n\times n$ complex matrix with the numerical range $W(A)$ satisfying
		$W(A)\subseteq\{re^{\pm i\theta}~:~\theta_1\leq\theta\leq\theta_2\},$
		where $r>0$ and $\theta_1,\theta_2\in \left[0,\pi/2\right],$ then 
		\begin{eqnarray*}
		&&(i)\,\,	w(A) \geq  \frac{csc\gamma}{2}\|A\|
			+ \frac{csc\gamma}{2}\left| \|\Im(A)\|-\|\Re(A)\|\right|,\,\,\text{and}\\
		&&(ii)\,\, w^2(A) \geq  \frac{csc^2\gamma}{4}\|AA^*+A^*A\|
			+ \frac{csc^2\gamma}{2}\left| \|\Im(A)\|^2-\|\Re(A)\|^2\right|,		\end{eqnarray*}
	where $\gamma=\max\{\theta_2,\pi/2-\theta_1\}$. We also prove that if $A,B$ are sectorial matrices with sectorial index $\gamma \in [0,\pi/2)$ and they are double commuting, then $w(AB)\leq \left(1+\sin^2\gamma\right)w(A)w(B).$
	
	 %Similar type of numerical radius inequalities for general sectorial matrices are also studied.
	
\end{abstract}

\section{{Introduction}}

\noindent Let  $\mathcal{M}_n$ denote the algebra of all $n\times n$ complex matrices. For $A\in  \mathcal{M}_n,$ the operator norm of $A$, denoted by $\|A\|$, is defined as $$\|A\|=\sup\{ \|Ax\| : x\in \mathbb{C}^n, \|x\|=1 \}.$$
  The numerical range and the numerical radius of $A$, denoted by $W(A)$ and $w(A)$ respectively,  are defined as  $$W(A)=\{ \langle Ax,x \rangle : x\in \mathbb{C}^n, \|x\|=1 \}$$ and $$w(A)=\sup\{ |\langle Ax,x \rangle| : x\in \mathbb{C}^n, \|x\|=1 \}.$$
  It is well known that the numerical radius $w(\cdot)$ defines a norm on $\mathcal{M}_n$, and $w(\cdot)$ is equivalent to the operator norm $\|\cdot\|.$ In fact, for each $ A\in \mathcal{M}_n,$ 
  \begin{eqnarray}\label{eqv}
  	\frac{1}{2}\|A\| \leq w(A)\leq \|A\| .
  \end{eqnarray}
  Note that $\frac{1}{2}\|A\| = w(A)$ if $A^2=0$, and $w(A)= \|A\|$ if $AA^*=A^*A$ (where $A^*$ denotes the adjoint of $A$).
 % For further basic properties of the numerical radius we refer to see the book \cite{Gus}. Due to the importance and various applications in matrix analysis of the numerical radius, many eminent researchers have devoted a considerable amount of time and effort to study it. 
 Computation of the exact value of the numerical radius $w(A)$ for  an arbitrary matrix $A$ is not an easy task, except for some special class of matrices. Therefore, for arbitrary matrices $A \in \mathcal{M}_n$, the researchers tried to develop upper and lower bounds of $w(A)$, which are sharper than the  bounds in \eqref{eqv}.  For recent developments of upper and lower bounds  of  the numerical radius we refer the readers to see  \cite{Bag_MIA_2020, bhunia_ADM_2021, Bhunia_RMJ_2021,Bhunia_RIM_2021,Bhunia_MIA_2021,Bhunia_BSM_2021,kitt_SM_2003,Omidver} and the references therein.  In \cite{kitt_LAA_2022,kitt_LAMA_2020}, the authors studied the numerical radius inequalities of a particular class of matrices, known as sectorial matrices. They developed several sharper upper and lower bounds of the numerical radius of sectorial matrices. Let us first mention the definition of a sectorial matrix.
 %Our aim in this article is to continue the study in that direction, i.e., to develop further sharper bounds of the numerical radius of sectorial matrices. Let us first introduce the following necessary notations and terminologies. 
  A matrix $A\in  \mathcal{M}_n$ is said to be positive if $\langle Ax,x\rangle\geq 0$ for all $x\in \mathbb{C}^n$. The positive matrix $A$ is said to be positive semi-definite (positive definite) if $\langle Ax,x\rangle\geq 0$ for all $x\in \mathbb{C}^n$ ($\langle Ax,x\rangle> 0$ for all non-zero $x\in \mathbb{C}^n$), and it is denoted by $A\geq 0$ ($A>0$).
  The Cartesian decomposition of $A\in  \mathcal{M}_n$ is $A=\Re(A)+i\Im(A)$, where $\Re(A)=\frac{A+A^*}{2}$ and $\Im(A)=\frac{A-A^*}{2i}.$
  A matrix $A\in  \mathcal{M}_n$ is said to be accretive if $\Re(A)>0.$ Clearly, all positive definite matrices are accretive.  A matrix $A\in \mathcal{M}_n$ is said to be accretive-dissipative if $\Re (A) > 0$ and $ \Im (A) >0.$ For a complex number $z$, let $\Re z $ and $\Im z$ denote the real and imaginary part of the complex number $z.$ Geometrically, a matrix $A\in  \mathcal{M}_n$ is accretive if and only if $W(A)\subseteq \{z \in \mathbb{C}: \Re z >0 \}.$
  \begin{defn}
   A matrix $A\in  \mathcal{M}_n$ is said to be sectorial if $W(A) \subseteq S_{\gamma}$ for some $\gamma \in [0, \frac{\pi}2)$, where
  $$ S_{\gamma}=\left\{z \in \mathbb{C}: \Re z >0, |\Im z|\leq (\tan \gamma) \Re z \right \}.$$  Geometrically, the numerical range of a sectorial matrix  lies entirely in a cone in the right half complex plane (i.e., $ \{z \in \mathbb{C}: \Re z >0 \}$), with vertex at the origin and half-angle  $\gamma$. Clearly, a positive definite matrix is a sectorial matrix with $\gamma=0.$ Henceforth, for simplicity, we write $A\in  \Pi_{\gamma}$ when $W(A) \subseteq S_{\gamma}$  for some $\gamma \in [0,\pi/2).$
  \end{defn}  
\begin{center}
	%\begin{figure}[h!]
	\begin{tikzpicture}
		\draw[->] (0,0)--(5,0);
		\draw[->] (0,0)--(0,4);
		\draw[->] (0,0)--(-4,0);
		\draw[->] (0,0)--(0,-4);
		%\draw[blue!80] (2,-2)--(2,2);
		%5\fill[fill=] (2,.5) circle (1);
		%\fill [fill=blue] (2,.5) ellipse (1.3cm and .7cm);
		\fill[rotate=30,blue] (2,-.46)circle(1.5cm and .9cm);
		\draw[] (0,0)--(3,3);
		\draw[] (0,0)--(3,-3);
		\node [] at (-.5,-.3) {$(0,0)$};
		\node [] at (-.3,.2) {O};
		\node [] at (3.2,3.1) {P};
		\node [] at (3.2,-3.1) {Q};
		\node [] at (4.5,-.25) {X};
		\node [] at (5.2,0) {$x$};
		\node [] at (0,4.2) {$y$};
		\fill[] (0,0) circle (1pt);
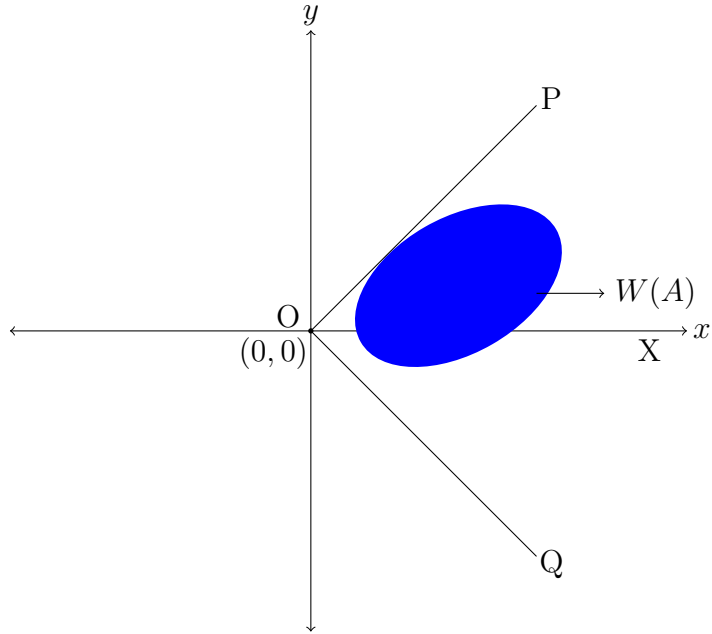
		\node [] at (0,-4.5) {Sectorial matrix $A$ with sectorial index $\angle POX=\angle QOX=\gamma$};
		\draw [->] [] (3,.5)--(3.9,.5) node[right]{$W(A)$};
		
	\end{tikzpicture}
	%	\end{figure}
\end{center}
Note that for a sectorial matrix,  $0$ does not belong to the numerical range. On the other hand if $0$ does not belong to the numerical range then there exists $\theta$ such that $ e^{i \theta} A$ is a sectorial matrix.  Such a matrix is known as rotational sectorial matrix. 
Next we observe that any sectorial matrix is an accretive matrix. Moreover, the converse is also true, i.e., if $A\in \mathcal{M}_n$ is an accretive matrix, then $A\in \Pi_{\gamma}$ for some $\gamma \in [0,\pi/2).$ 
  For $A\in \Pi_{\gamma}$,  the following inequality  
  \begin{eqnarray}\label{eqvsec}
  	\cos \gamma \|A\| \leq w(A)	
  \end{eqnarray} 
  holds (see \cite[Prop. 3.1]{kitt_LAMA_2020}).
  We would like to note that the inequality \eqref{eqvsec} gives better bound than the first bound in \eqref{eqv}, when $\gamma \in \left[0, {\pi}/{3}\right).$
  
   In this article we present several numerical radius inequalities of the sectorial matrices, which improve on the existing  inequalities.  Further, we develop numerical radius inequalities of the commutator and anti-commutator of  matrices in $\mathcal{M}_n.$  
   
    Before we end this section we introduce  the fractional powers of matrices via Dunford-Taylor integral. 
   Let $f: \mathcal{D} \to \mathbb{C}$ be an analytic function defined on a domain $\mathcal{D}$. Let $A\in \mathcal{M}_n$ and let the spectrum of $A,$  $\sigma (A)$  (i.e., the eigenvalues of $A$),  lies in $\mathcal{D}$. Then by using Dunford-Taylor integral, $f(A)$ can be written  as 
  \begin{eqnarray}\label{int1}
  	f(A)= \frac{1}{2 \pi i}\int_{\Gamma} f(z)(zI-A)^{-1}dz,
  \end{eqnarray} 
where  $ \Gamma \subset \mathcal{D}$   be a simple closed smooth curve
with positive direction enclosing all eigenvalues of $A$ in its interior (see  \cite[p. 44]{Kato} and \cite[p. 287]{Taylor}). If $A$ is an accretive matrix then $A$ has no eigenvalues in $(-\infty, 0]$ and so 
 
  \begin{eqnarray}\label{int2}
  	A^t= \frac{1}{2 \pi i}\int_{\Gamma} z^t(zI-A)^{-1}dz,
  \end{eqnarray}
  where $t\in (0,1)$ and $z^t$ is the principal branch of the multi-valued function $z\to z^t$, and $\Gamma$ is a suitable curve. 
Thus for $0 < t < 1,$  $A^t$ is sectorial whenever $A$ is sectorial. 
  Note that for $t>1$, it is not guaranteed that $A^t$ is sectorial when $A$ is sectorial. Moreover (see  \cite{drury}), if $A\in \Pi_{\gamma}$, then 
  \begin{eqnarray}\label{tsec}
  	A^t\in \Pi_{t\gamma}, \,\,0<t<1.
  \end{eqnarray} 
For further related discussion  we refer to \cite{kitt_Pos_2021,Masser}.

\section{{Main results}}

We begin with the following lemma (see \cite{kitt_LAA_2022}), proof follows from the Cartesian decomposition of $A\in  \Pi_{\gamma}$ and the inequality  $|\langle\Im(A)x,x\rangle|\leq (\tan \gamma) \langle \Re(A)x,x\rangle$ for all $x\in \mathbb{C}^n$ with $\|x\|=1$.
	
	\begin{lemma}\label{lemma1}  $($\cite{kitt_LAA_2022}$)$
		Let $A \in \Pi_{\gamma}$. Then 
		\begin{eqnarray*}
			\|\Im(A)\| \leq (sin \gamma)\, w(A).
		\end{eqnarray*} 
	\end{lemma}

Now, we are in a position to present our first inequality.
	
	\begin{theorem}\label{th1}
		Let $A \in \Pi_{\gamma}$  with $ \gamma \neq 0.$ Then 
		\begin{eqnarray}\label{th1eq0}
			w^2(A) \geq  \frac{csc^2\gamma}{4}\|AA^*+A^*A\|
			+ \frac{csc^2\gamma}{2}\left( \|\Im(A)\|^2-\|\Re(A)\|^2\right).
		\end{eqnarray}
	\end{theorem}
    \begin{proof}
    	Let $x\in \mathbb{C}^n$ be such that $\|x\|=1$. From the Cartesian decomposition of $A,$ we have, $|\langle Ax,x \rangle|^2=\langle \Re(A)x,x \rangle^2+\langle \Im(A)x,x \rangle^2.$ This gives that 
    	\begin{eqnarray}\label{th1eq1}
    		w^2(A) \geq \|\Re(A)\|^2.
    	\end{eqnarray}
    	Also, Lemma \ref{lemma1} gives that 
    	\begin{eqnarray}\label{th1eq2}
    	w^2(A) \geq csc^2\gamma\|\Im(A)\|^2.
    	\end{eqnarray}
    	Therefore, from the inequalities (\ref{th1eq1}) and (\ref{th1eq2}) we have, 
    	\begin{eqnarray*}
    		 w^2(A) 
    		& \geq&  \max\{\|\Re(A)\|^2,csc^2\gamma\|\Im(A)\|^2\}\\
    		&=&\frac{\|\Re(A)\|^2+csc^2\gamma\|\Im(A)\|^2}{2}+\frac{\left|\|\Re(A)\|^2-csc^2\gamma\|\Im(A)\|^2\right|}{2}\\
    		&=&\frac{csc^2\gamma}{2}\left(\|\Re^2(A)\|+\|\Im^2(A)\|\right)+\frac{1-csc^2\gamma}{2}\|\Re(A)\|^2\\
    		&& +\left|\frac{\|\Re(A)\|^2-csc^2\gamma\|\Im(A)\|^2}{2}\right|\\
    		& \geq& \frac{csc^2\gamma}{2}\|\Re^2(A)+\Im^2(A)\|+\frac{1-csc^2\gamma}{2}\|\Re(A)\|^2\\
    		&&+\left|\frac{\|\Re(A)\|^2-csc^2\gamma\|\Im(A)\|^2}{2}\right|\\
    		& \geq& \frac{csc^2\gamma}{4}\|AA^*+A^*A\| -\frac{csc^2\gamma}{2}\|\Re(A)\|^2\\
    		&&+\left|\frac{\|\Re(A)\|^2}{2}-\frac{\|\Re(A)\|^2-csc^2\gamma\|\Im(A)\|^2}{2}\right|\\
    		&=& \frac{csc^2\gamma}{4}\|AA^*+A^*A\| -\frac{csc^2\gamma}{2}\|\Re(A)\|^2+\frac{csc^2\gamma}{2}\|\Im(A)\|^2\\
    		&=& \frac{csc^2\gamma}{4}\|AA^*+A^*A\|
    		+ \frac{csc^2\gamma}{2}\left( \|\Im(A)\|^2-\|\Re(A)\|^2\right).
    	\end{eqnarray*}
    This completes the proof.
    \end{proof}

\begin{remark}\label{rmk1}
	%\begin{enumerate}
	(i) Let $A\in \Pi_{\gamma}\, (\gamma \neq 0).$ If  $\|\Im(A)\|\geq \|\Re(A)\|$ then  from Theorem \ref{th1} we get
	\begin{eqnarray*}
		 w^2(A) &\geq& \frac{csc^2\gamma}{4}\|AA^*+A^*A\|
		 + \frac{csc^2\gamma}{2}\left( \|\Im(A)\|^2-\|\Re(A)\|^2\right)\\
		 &>& \frac{1}{4}\|AA^*+A^*A\|
		 + \frac{1}{2}\left( \|\Im(A)\|^2-\|\Re(A)\|^2\right)\,\left(\textit{$csc\gamma > 1$ $\forall$ $\gamma \in [0,\pi/2)$}\right)\\
		 & \geq &  \frac{1}{4}\|AA^*+A^*A\|.
	\end{eqnarray*}
    Thus in this case  Theorem \ref{th1} gives sharper bound than the well known bound \cite[Th. 1]{kitt_SM_2005}, namely,
    \begin{eqnarray}\label{rmk1eq1}
    	w^2(A) \geq \frac{1}{4}\|A^*A+AA^*\|.
    \end{eqnarray}
(ii) Let $A\in \Pi_{\gamma} \, (\gamma \neq 0)$ with $\|\Im(A)\|< \|\Re(A)\|.$ Then Theorem \ref{th1} gives sharper bound than the bound (\ref{rmk1eq1}) if $sin\gamma < \sqrt{1-\frac{\|\Re(A)\|^2-\|\Im(A)\|^2}{\|\Re^2(A)+\Im^2(A)\|}}.$  For example, let $A=\left( \begin{matrix}
	3+2i & 0\\
	0& 1
\end{matrix}\right)$. Then, $A\in \Pi_{\gamma}$ where $\sin \gamma =2/\sqrt{13}< \sqrt{1-\frac{\|\Re(A)\|^2-\|\Im(A)\|^2}{\|\Re^2(A)+\Im^2(A)\|}}= 2\sqrt2/\sqrt{13}.$ 
Theorem \ref{th1} gives $w^2(A)\geq 13$  whereas   (\ref{rmk1eq1}) gives $ w^2(A) \geq 13/2$. \\
 (iii) For any $A \in \mathcal{M}_n,$  it was proved in \cite[Th. 2.9]{bhunia_LAA_2021} that
\begin{eqnarray}\label{rmk1eq3}
	w^2(A)\geq\frac{1}{4}\|A^*A+AA^*\|+\frac{1}{2}|\|\Re(A)\|^2-\|\Im(A)\|^2|.	
\end{eqnarray}
We would like to remark that when $A \in \Pi_{\gamma}$ with $\|\Im(A)\|\geq \|\Re(A)\|$,  Theorem \ref{th1} gives sharper bound than the bound (\ref{rmk1eq3}).
%\end{enumerate}
\end{remark}

\vspace{.2cm}

Now, by applying Theorem \ref{th1} we develop upper bounds for the numerical radius of generalized commutator and anti-commutator matrices. Recall that for $A,B\in \mathcal{M}_n$, the matrix $AB-BA$ is called commutator matrix and $AB+BA$ is called anti-commutator matrix.

	\begin{theorem}\label{th2}
      Let $A,B, X,Y \in \mathcal{M}_n$ with $A \in \Pi_{\gamma} \, ( \gamma \neq 0).$ Then 
    	\begin{eqnarray*}
    		&&w(AXB\pm BYA)\\ 
    		&&\leq 2\sqrt{2}(sin\gamma) \|B\|\max\{\|X\|,\|Y\|\} \sqrt{w^2(A)-\frac{csc^2\gamma}{2}\left( \|\Im(A)\|^2-\|\Re(A)\|^2\right)}.
    	\end{eqnarray*}
   \end{theorem}

   \begin{proof}
   	Let $x\in \mathbb{C}^n$ with $\|x\|=1$. Suppose $\|X\|,\|Y\| \leq 1.$ Then, by Cauchy-Schwarz inequality we have,
   	      \begin{eqnarray}\label{th2eq1}
   	          |\langle (AX \pm YA)x,x \rangle| & \leq & |\langle AX x,x \rangle|+ |\langle  YA x,x \rangle|\nonumber\\
   	          &=& |\langle X x,A^* x \rangle|+ |\langle  A x,Y^* x \rangle|\nonumber\\
   	          & \leq & \|A^*x\|+\|A^*x\|\nonumber\\
   	          & \leq & \sqrt{2(\|A^*x\|^2+\|A^*x\|^2)}\,\,\,~~\mbox{\big(\textit{by convexity of $f(x)=x^2$}\big)}\nonumber\\
   	          & \leq & \sqrt{2\|AA^*+A^*A\|}.
   	      \end{eqnarray}
   	      Now, it follows from Theorem \ref{th1} that
   	      \begin{eqnarray}\label{th2eq2}
   	         \|AA^*+A^*A\| \leq 4sin^2\gamma\left(w^2(A)-\frac{csc^2\gamma}{2}\left( \|\Im(A)\|^2-\|\Re(A)\|^2\right)\right).
   	      \end{eqnarray}
   	     Therefore, by using the inequality (\ref{th2eq2}) in (\ref{th2eq1}) we have,
   	     \begin{eqnarray*}
   	        |\langle (AX \pm YA)x,x \rangle| & \leq & 2\sqrt{2} (sin\gamma)\sqrt{w^2(A)-\frac{csc^2\gamma}{2}\left( \|\Im(A)\|^2-\|\Re(A)\|^2\right)}.
   	     \end{eqnarray*}
        Taking supremum over $x \in \mathbb{C}^n,$ $\|x\|=1$ we have,
        \begin{eqnarray}\label{th2eq3}
        w(AX \pm YA)\leq 2\sqrt{2}(sin\gamma)\sqrt{w^2(A)-\frac{csc^2\gamma}{2}\left( \|\Im(A)\|^2-\|\Re(A)\|^2\right)}.
        \end{eqnarray}
        Now, Theorem \ref{th2} holds trivially when $X=Y=0.$ Let $\max\{\|X\|,\|Y\|\} \neq 0.$ Then $\left\|\frac{X}{\max\{\|X\|,\|Y\|\}}\right\|\leq 1$ and $\left\|\frac{Y}{\max\{\|X\|,\|Y\|\}}\right\|\leq 1.$ Therefore, from (\ref{th2eq3}) 
        %(replacing $X$ and $Y$ by $\frac{X}{\max\{\|X\|,\|Y\|\}}$ and $\frac{Y}{\max\{\|X\|,\|Y\|\}},$ respectively) 
        we have,
        \begin{eqnarray*} %\label{th2eq4}
        && w(AX \pm YA) \\
        &&\leq 2\sqrt{2}(sin\gamma) \max\{\|X\|,\|Y\|\}\sqrt{w^2(A)-\frac{csc^2\gamma}{2}\left( \|\Im(A)\|^2-\|\Re(A)\|^2\right)}.
        \end{eqnarray*}
        Replacing $X$ and $Y$ by $XB$ and $BY,$ respectively, we get,
        \begin{eqnarray*}
        &&w(AXB \pm BYA) \\
        &&\leq 2\sqrt{2}(sin\gamma) \max\{\|XB\|,\|BY\|\}\sqrt{w^2(A)-\frac{csc^2\gamma}{2}\left( \|\Im(A)\|^2-\|\Re(A)\|^2\right)}\\
        && \leq 2\sqrt{2} (sin\gamma) \|B\|\max\{\|X\|,\|Y\|\} \sqrt{w^2(A)-\frac{csc^2\gamma}{2}\left( \|\Im(A)\|^2-\|\Re(A)\|^2\right)}.
        \end{eqnarray*}
    This completes the proof.
   	\end{proof}

The following corollary follows from Theorem \ref{th2} by taking $X=Y=I.$

\begin{cor}\label{th2cor1}
	Let $A, B \in \mathcal{M}_n$ with $A \in \Pi_{\gamma}\, ( \gamma \neq 0).$ Then 
	\begin{eqnarray*}
		w(AB \pm BA) \leq 2\sqrt{2} (sin\gamma) \|B\| \sqrt{w^2(A)-\frac{csc^2\gamma}{2}\left( \|\Im(A)\|^2-\|\Re(A)\|^2\right)}.
	\end{eqnarray*}
\end{cor}

\begin{remark}\label{rmk2}
		(i) Let $A, B \in \mathcal{M}_n$ with $A\in \Pi_{\gamma}\, ( \gamma \neq 0)$ and  $\|\Im(A)\|\geq \|\Re(A)\|.$ Then from Corollary \ref{th2cor1} we have,
	\begin{eqnarray*}
		w(AB \pm BA) &\leq& 2\sqrt{2}( sin\gamma) \|B\| \sqrt{w^2(A)-\frac{csc^2\gamma}{2}\left( \|\Im(A)\|^2-\|\Re(A)\|^2\right)}\\
		 &\leq& 2\sqrt{2}( sin\gamma) \|B\| \sqrt{w^2(A)-\frac{1}{2}\left( \|\Im(A)\|^2-\|\Re(A)\|^2\right)}\\
		 &<& 2\sqrt{2} \|B\| \sqrt{w^2(A)-\frac{1}{2}\left( \|\Im(A)\|^2-\|\Re(A)\|^2\right)}\\
		&\leq&2\sqrt{2} \|B\|w(A),
	\end{eqnarray*}
	as $sin\gamma < 1$ for all $\gamma \in (0,\pi/2).$ Thus, for any $A \in \Pi_{\gamma}\, ( \gamma \neq 0),$ Corollary \ref{th2cor1} gives sharper bound than the well known bound \cite[Th. 11]{fong_1983}, namely,
	\begin{eqnarray}\label{rmk2eq1}
		w(AB \pm BA) \leq 2\sqrt{2} \|B\|w(A).
	\end{eqnarray}
	(ii) Let $A\in \Pi_{\gamma} \,(\gamma \neq 0)$ with $\|\Im(A)\|< \|\Re(A)\|.$ Then Corollary \ref{th2cor1} gives better bound than the bound (\ref{rmk2eq1}) when $cos^2\gamma >\frac{\|\Re(A)\|^2-\|\Im(A)\|^2}{2w^2(A)}$.\\
	(iii) 
	 For any $A,B \in \mathcal{M}_n,$ it was proved in  \cite[Cor. 3.4]{kitt_MS_2014} that
	\begin{eqnarray}\label{rmk2eq3}
		w(AB \pm BA) \leq 2\sqrt{2}\|B\|\sqrt{w^2(A)-\frac{\left|\|\Re(A)\|^2-\|\Im(A)\|^2\right|}{2}}.	
	\end{eqnarray}
	We would like to remark that Corollary \ref{th2cor1} gives sharper bound than
	the bound (\ref{rmk2eq3}) when $A\in \Pi_{\gamma}$ with $\|\Im(A)\|\geq \|\Re(A)\|$.
\end{remark}

On the basis of  Corollary \ref{th2cor1} we obtain the following inequality.

\begin{cor}\label{th2cor2}
	Let $A,B \in \Pi_{\gamma}\, ( \gamma \neq 0).$ Then 
	 \begin{eqnarray}\label{th2cor2eq1}
		w(AB \pm BA) \leq \min\{\beta_1,\beta_2\},
	\end{eqnarray}
 where $$\beta_1=2\sqrt{2} (sin\gamma) \|B\| \sqrt{w^2(A)-\frac{csc^2\gamma}{2}\left( \|\Im(A)\|^2-\|\Re(A)\|^2\right)}$$ and 
 $$\beta_2=2\sqrt{2} (sin\gamma) \|A\| \sqrt{w^2(B)-\frac{csc^2\gamma}{2}\left( \|\Im(B)\|^2-\|\Re(B)\|^2\right)}.$$
\end{cor} 

\begin{proof}
	The proof follows from Corollary \ref{th2cor1} by interchanging $A$ and $B$. 
\end{proof}

%\vspace{.2cm}

Next, in the following proposition  we obtain a norm inequality for the sum of $n$ matrices which have no eigenvalues in  $(-\infty,0]$.

\begin{proposition}\label{prop1}
	Let $A_i\in \mathcal{M}_n$ with  no eigenvalues in $(-\infty,0]$ for $i=1,\ldots,n.$ Then  
	\begin{eqnarray*}
		\left\|\sum_{i=1}^{n}A_i\right\|^{1/2} \leq \sum_{i=1}^{n}\|A_i^{1/2}\|.
	\end{eqnarray*}
\end{proposition}

\begin{proof}
	We have,
		\begin{eqnarray*}
		\left\|\sum_{i=1}^{n}A_i\right\|
		= && 
		\left\| \begin{pmatrix}
			 \sum_{i=1}^{n}A_i&0&\ldots&0\\
			0&0&\ldots&0\\
			\vdots\\
			0&0&\ldots&0
		\end{pmatrix}\right\|\\
		= &&\left\|\begin{pmatrix}
		A^{1/2}_1&A^{1/2}_2&\ldots&A^{1/2}_n\\
		0&0&\ldots&0\\
		\vdots\\
		0&0&\ldots&0
	\end{pmatrix}
	\begin{pmatrix}
		A^{1/2}_1&0&\ldots&0\\
		A^{1/2}_2&0&\ldots&0\\
		\vdots\\
		A^{1/2}_n&0&\ldots&0
	\end{pmatrix}\right\|\\
     \leq && \left\|\begin{pmatrix}
    	A^{1/2}_1&A^{1/2}_2&\ldots&A^{1/2}_n\\
    	0&0&\ldots&0\\
    	\vdots\\
    	0&0&\ldots&0
    \end{pmatrix}\right\|
    \left\|\begin{pmatrix}
    	A^{1/2}_1&0&\ldots&0\\
    	A^{1/2}_2&0&\ldots&0\\
    	\vdots\\
    	A^{1/2}_n&0&\ldots&0
    \end{pmatrix}\right\|\\
     \leq && \left(\left\|\begin{pmatrix}
   	A^{1/2}_1&0&\ldots&0\\
   	0&0&\ldots&0\\
   	\vdots\\
   	0&0&\ldots&0
   \end{pmatrix}\right\|+\ldots+\left\|\begin{pmatrix}
   0&0&\ldots&A^{1/2}_n\\
   0&0&\ldots&0\\
   \vdots\\
   0&0&\ldots&0
\end{pmatrix}\right\|\right)\\
   && \times  \left(\left\|\begin{pmatrix}
   	A^{1/2}_1&0&\ldots&0\\
   	0&0&\ldots&0\\
   	\vdots\\
   	0&0&\ldots&0
   \end{pmatrix}\right\|+\ldots+\left\|\begin{pmatrix}
   	0&0&\ldots&0\\
   	0&0&\ldots&0\\
   	\vdots\\
   	A^{1/2}_n&0&\ldots&0
   \end{pmatrix}\right\|\right)\\
  =  &&  \left(\sum_{i=1}^{n}\|A_i^{1/2}\|\right)^2.
	\end{eqnarray*}
 This completes the proof.
\end{proof}

Note that the inequality \cite[Lemma 3.2]{kitt_LAA_2022} follows from Proposition \ref{prop1} (for  $n=2$).   
Now, by using Proposition \ref{prop1} we obtain the following corollary, and for this first we note that the following lemma (see \cite{kitt_LAA_2022}) which  follows from the norm inequalities (see  \cite{Garling,Mirman}):
 $\|A\|^2\leq \|\Re(A)\|^2+2\|\Im(A)\|^2$ when $A\in \mathcal{M}_n$ is accretive and $\|A\|^2\leq \|\Re(A)\|^2+\|\Im(A)\|^2$ when $A\in \mathcal{M}_n$ is accretive-dissipative.
 % (i.e., $\Re(A)>0$ and $\Im(A)>0$).

\begin{lemma}\label{lemma3}  $($\cite{kitt_LAA_2022}$)$
	Let $A \in \Pi_{\gamma}.$ Then 
	\begin{eqnarray*}
		\|A\| \leq \sqrt{1+2sin^2\gamma} \,w(A).
	\end{eqnarray*} 
	Moreover, if $A \in \Pi_{\gamma}$ is accretive-dissipative, then
	\begin{eqnarray*}
		\|A\| \leq \sqrt{1+sin^2\gamma}\, w(A).
	\end{eqnarray*}  
\end{lemma}

\begin{cor}\label{prop1cor1}
	Let $A_i\in \Pi_{\gamma}$ for $i=1,\ldots,n.$ Then
	\begin{eqnarray*}
		w^{1/2}\left(\sum_{i=1}^{n}A_i\right) \leq \sqrt{1+2sin^2\gamma/2} \,\sum_{i=1}^{n}w(A_i^{1/2}).
	\end{eqnarray*}
Moreover, if $A_i$ is accretive-dissipative for $i=1,\ldots,n,$ then
\begin{eqnarray*}
	w^{1/2}\left(\sum_{i=1}^{n}A_i\right) \leq \sqrt{1+sin^2\gamma/2}\,\sum_{i=1}^{n}w(A_i^{1/2}).
\end{eqnarray*}
\end{cor}

\begin{proof}
	Using Proposition \ref{prop1},  Lemma \ref{lemma3} and \eqref{tsec} we have, 
	\begin{eqnarray*}
		w^{1/2}\left(\sum_{i=1}^{n}A_i\right)
		 \leq  \left\|\sum_{i=1}^{n}A_i\right\|^{1/2}
		 \leq \sum_{i=1}^{n}\|A_i^{1/2}\| \leq \sqrt{1+2sin^2\gamma/2}\,\sum_{i=1}^{n}w(A_i^{1/2}).
	\end{eqnarray*}
   Similarly, for accretive-dissipative matrices $A_i$ for $i=1,\ldots,n,$ we have,
   \begin{eqnarray*}
   	w^{1/2}\left(\sum_{i=1}^{n}A_i\right)
   	\leq  \left\|\sum_{i=1}^{n}A_i\right\|^{1/2}
   	\leq \sum_{i=1}^{n}\|A_i^{1/2}\| \leq \sqrt{1+sin^2\gamma/2}\,\sum_{i=1}^{n}w(A_i^{1/2}).
   \end{eqnarray*}
   
\end{proof}

%Note that for $n=2,$ we have \cite[Cor. 3.6]{kitt_LAA_2022}. 
From Corollary \ref{prop1cor1} (for $n=1$) we have the following inequality, which is   also given in \cite{kitt_LAA_2022}.

\begin{cor}\label{lemma5}  %$($\cite{kitt_LAA_2022}$)$
	Let $A \in\Pi_{\gamma}.$ Then $$w^{1/2}(A) \leq \sqrt{1+2sin^2\gamma/2}\,w(A^{1/2}).$$ Further, if $A$ is accretive-dissipative, then $$w^{1/2}(A) \leq \sqrt{1+sin^2\gamma/2}\,w(A^{1/2}).$$
\end{cor}

From the above corollary we prove the following power inequality.

\begin{theorem}\label{cor1}
	Let $A \in\Pi_{\gamma}.$ Then 
	\begin{eqnarray}\label{cor1eq1}
		w^{1/2^n}(A) \leq \Pi_{i=1}^{n}\left(1+2sin^2\gamma/2^i\right)^{\frac{1}{2^{n+1-i}}}w(A^{1/2^n}),
	\end{eqnarray}
	for all $n \in \mathbb{N}.$  
	Further, if $A$ is accretive-dissipative, then
	\begin{eqnarray}\label{cor1eq2}
		w^{1/2^n}(A) \leq \Pi_{i=1}^{n}\left(1+sin^2\gamma/2^i\right)^{\frac{1}{2^{n+1-i}}}w(A^{1/2^n}),
	\end{eqnarray}
	for all $n \in \mathbb{N}.$ 
\end{theorem}

\begin{proof}
	We prove the inequality (\ref{cor1eq1}) by mathematical induction. It is clear form Corollary \ref{lemma5} that the inequality (\ref{cor1eq1}) holds for $n=1.$ Now, we assume that the inequality (\ref{cor1eq1}) holds for $n=k\, (k\geq 1),$ i.e., 
	\begin{eqnarray}\label{cor1eq3}
		w^{1/2^k}(A) \leq \Pi_{i=1}^{k}\left(1+2sin^2\gamma/2^i\right)^{\frac{1}{2^{k+1-i}}}w(A^{1/2^k}).
	\end{eqnarray}
	By \eqref{tsec} we have, $A^{1/2^k}\in \Pi_{\gamma/2^k}.$ So, it follows from Corollary \ref{lemma5} that
	\begin{eqnarray}\label{cor1eq4}
		w(A^{1/2^k}) \leq (1+2sin^2\gamma/2^{k+1}) w^2(A^{1/2^{k+1}}).
	\end{eqnarray}
	From (\ref{cor1eq3}) and (\ref{cor1eq4}), we have
	\begin{eqnarray*}
		w^{1/2^k}(A) &\leq& \Pi_{i=1}^{k}\left(1+2sin^2\gamma/2^i\right)^{\frac{1}{2^{k+1-i}}}(1+2sin^2\gamma/2^{k+1}) w^2(A^{1/2^{k+1}})\\
		&=& \Pi_{i=1}^{k+1}\left(1+2sin^2\gamma/2^i\right)^{\frac{1}{2^{k+1-i}}} w^2(A^{1/2^{k+1}}).
	\end{eqnarray*}
	Therefore,
	\begin{eqnarray*}
		w^{1/2^{k+1}}(A) \leq \Pi_{i=1}^{k+1}\left(1+2sin^2\gamma/2^i\right)^{\frac{1}{2^{k+2-i}}} w(A^{1/2^{k+1}}).
	\end{eqnarray*}
	Therefore, the inequality (\ref{cor1eq1}) holds for $n=k+1.$ Hence the inequality (\ref{cor1eq1}) holds  for all $n \in \mathbb{N}.$ Proceeding similarly as above, and using Corollary \ref{lemma5} for accretive-dissipative matrices, we get the desired
	inequality (\ref{cor1eq2}). 
\end{proof}

Next proposition reads as follows:

\begin{proposition}\label{prop2}
	Let $A,B\in \mathcal{M}_n$ with  no eigenvalues in $(-\infty,0].$ Then  
	\begin{eqnarray*}
		(i)~~\|A+B\| &\leq& (\|A^{\alpha}\|+\|B^{\alpha}\|)(\|A^{1-\alpha}\|+\|B^{1-\alpha}\|),\\
		(ii)~~\|A+B\| &\leq& (\|A^{\alpha}\|+\|B^{1-\alpha}\|)(\|A^{1-\alpha}\|+\|B^{\alpha}\|),
	\end{eqnarray*}
for all $\alpha \in (0,1).$
\end{proposition}

\begin{proof}
	(i) We have,
	\begin{eqnarray*}
	       \|A+B\| &=& \left\|\begin{pmatrix}
	       	A+B&0\\
	       	0&0
	       \end{pmatrix}\right\|\\
       &=&  \left\|\begin{pmatrix}
       	A^{\alpha}&B^{\alpha}\\
       	0&0
       \end{pmatrix}\begin{pmatrix}
       A^{1-\alpha}&0\\
       B^{1-\alpha}&0
   \end{pmatrix}\right\|\\
     &\leq & \left\|\begin{pmatrix}
     	A^{\alpha}&B^{\alpha}\\
     	0&0
     \end{pmatrix}\right\|\left\|\begin{pmatrix}
     	A^{1-\alpha}&0\\
     	B^{1-\alpha}&0
     \end{pmatrix}\right\|\\
    & \leq & \left(\left\|\begin{pmatrix}
    	A^{\alpha}&0\\
    	0&0
    \end{pmatrix}\right\|+\left\|\begin{pmatrix}
    0&B^{\alpha}\\
    0&0
\end{pmatrix}\right\|\right)\left(\left\|\begin{pmatrix}
A^{1-\alpha}&0\\
0&0
\end{pmatrix}\right\|+\left\|\begin{pmatrix}
0&0\\
B^{1-\alpha}&0
\end{pmatrix}\right\|\right)\\
&=& (\|A^{\alpha}\|+\|B^{\alpha}\|)(\|A^{1-\alpha}\|+\|B^{1-\alpha}\|).
    \end{eqnarray*}

\noindent (ii)  We have,
\begin{eqnarray*}
	\|A+B\| &=& \left\|\begin{pmatrix}
		A+B&0\\
		0&0
	\end{pmatrix}\right\|\\
	&=&  \left\|\begin{pmatrix}
		A^{\alpha}&B^{1-\alpha}\\
		0&0
	\end{pmatrix}\begin{pmatrix}
		A^{1-\alpha}&0\\
		B^{\alpha}&0
	\end{pmatrix}\right\|\\
	&\leq & \left\|\begin{pmatrix}
		A^{\alpha}&B^{1-\alpha}\\
		0&0
	\end{pmatrix}\right\|\left\|\begin{pmatrix}
		A^{1-\alpha}&0\\
		B^{\alpha}&0
	\end{pmatrix}\right\|\\
	& \leq & \left(\left\|\begin{pmatrix}
		A^{\alpha}&0\\
		0&0
	\end{pmatrix}\right\|+\left\|\begin{pmatrix}
		0&B^{1-\alpha}\\
		0&0
	\end{pmatrix}\right\|\right)\left(\left\|\begin{pmatrix}
		A^{1-\alpha}&0\\
		0&0
	\end{pmatrix}\right\|+\left\|\begin{pmatrix}
		0&0\\
		B^{\alpha}&0
	\end{pmatrix}\right\|\right)\\
	&=& \left(\|A^{\alpha}\|+\|B^{1-\alpha}\|)(\|A^{1-\alpha}\|+\|B^{\alpha}\|\right).
\end{eqnarray*}

\end{proof}

Based on Proposition \ref{prop2} we obtain the following corollary.

\begin{cor}\label{prop2cor1}
	Let $A,B \in\Pi_{\gamma}.$ Then for all $\alpha \in (0,1),$
	\begin{eqnarray*}
		 w(A+B)
		& \leq&  \sqrt{1+2sin^2\alpha\gamma} \sqrt{1+2sin^2(1-\alpha)\gamma}\\
		&& \times \left(w(A^\alpha)+w(B^\alpha)\right)\left(w(A^{1-\alpha})+w(B^{1-\alpha})\right)
	\end{eqnarray*}
and
		\begin{eqnarray*} w(A+B)
		& \leq& (\sqrt{1+2sin^2\alpha\gamma}\,w(A^\alpha)+ \sqrt{1+2sin^2(1-\alpha)\gamma}\,w(B^{1-\alpha}))\\
		&& \times  (\sqrt{1+2sin^2(1-\alpha)\gamma}\,w(A^{1-\alpha})+ \sqrt{1+2sin^2\alpha\gamma}\,w(B^{\alpha})).
	\end{eqnarray*}
   
\end{cor}

Next,  we obtain the numerical radius inequalities of the sum and the product of double commuting sectorial matrices.  Recall that a set of matrices $\{A_i : i=1,2, \ldots,n \} \subseteq \mathcal{M}_n$ is called double commuting if $A_iA_j=A_jA_i$ and $A_iA_j^*=A_j^*A_i$ for all $i\neq j.$  First, we note the following lemma, proved in \cite{naja_LAA_2020}.

\begin{lemma}\label{lemma4}  %$($\cite{naja_LAA_2020}$)$
	Let $\{T_i, S_i : i =1,2,\dots, n \} \subseteq \mathcal{M}_n$  be a set of double commuting matrices.  Then
	\begin{eqnarray*}
	      w\left(\sum_{i=1}^{n}A_iB_i\right) \leq \frac{1}{2}\left\|\sum_{i=1}^{n}A^*_iA_i+A_iA^*_i\right\|^{1/2}\left\|\sum_{i=1}^{n}B^*_iB_i+B_iB^*_i\right\|^{1/2}.
	\end{eqnarray*}
\end{lemma}

Applying the above lemma we prove the following theorem. 

\begin{theorem}\label{th3}
	Let $\{T_i, S_i : i =1,2,\dots,n \}\subseteq \mathcal{M}_n$  be a set of double commuting sectorial matrices.  Then
	\begin{eqnarray*}
		 w\left(\sum_{i=1}^{n}A_iB_i\right) \leq (1+sin^2\gamma)\left(\sum_{i=1}^{n}w^2(A_i)\right)^{1/2}\left(\sum_{i=1}^{n}w^2(B_i)\right)^{1/2}.
	\end{eqnarray*}
\end{theorem}

\begin{proof}
 From  Lemma \ref{lemma4} we get,
	\begin{eqnarray*}
		 && w^2\left(\sum_{i=1}^{n}A_iB_i\right) \\
		 &&\leq \left\|\sum_{i=1}^{n}(\Re^2(A_i)+\Im^2(A_i))\right\|\left\|\sum_{i=1}^{n}(\Re^2(B_i)+\Im^2(B_i))\right\|\\
		 && \leq \sum_{i=1}^{n}(\|\Re(A_i)\|^2+\|\Im(A_i)\|^2)\sum_{i=1}^{n}(\|\Re(B_i)\|^2+\|\Im(B_i)\|^2)\\
		 && \leq (1+sin^2\gamma)^2\left(\sum_{i=1}^{n}w^2(A_i)\right)\left(\sum_{i=1}^{n}w^2(B_i)\right) \,\,\,~~\mbox{\big(\textit{by Lemma \ref{lemma1}}\big)}.
	\end{eqnarray*} 
Therefore, we get the required result.
\end{proof}

From Theorem \ref{th3} (for $n=1$) we obtain the following corollary.

\begin{cor}\label{th3cor1}
	Let $A,B \in\Pi_{\gamma}$ be such that $AB=BA$ and $AB^*=B^*A.$ Then
	\begin{eqnarray*}
		w(AB) \leq (1+sin^2\gamma)w(A)w(B).
	\end{eqnarray*}
\end{cor}

\begin{remark}
	Let  $A,B \in\Pi_{\gamma}$ be such that $AB=BA$ and $AB^*=B^*A.$ Then
	it is clear from Corollary \ref{th3cor1} that the upper bound of $w(AB)$ depends on $\gamma$, and interpolates between $w(A)w(B)$ and $2w(A)w(B).$ If $A,B$ are positive definite (i.e., $A,B>0$), then $w(AB) \leq w(A)w(B)$ and if $\gamma \to \frac{\pi^-}{2}$, then $w(AB) \leq 2w(A)w(B).$
\end{remark}

  %{\color{red}At the end of this section, we recall the rotational sectorial matrices. A matrix $A\in \mathcal{M}_n$ is said to be rotational sectorial if there exists  $\theta \in [0, \pi]$ such that $e^{i\theta}A$ is sectorial. If  $A\in \mathcal{M}_n$ is rotational sectorial such that 
% $e^{i\theta}A\in \Pi_{\gamma}$, then we write $A\in \Pi_{\theta,\gamma}.$
%  The convexity of the numerical range implies that when $0\notin W(A)$, the numerical range $W(A)$ contained in a sector center at the origin, say, $\{re^{i\theta} : \, \theta_1\leq \theta \leq \theta_2  \}$ with $r> 0$ and $\theta_2-\theta_1< \pi.$ In this case, there exists $\theta \in [0, \pi]$ such that $e^{i\theta}A\in \Pi_{\gamma}$, where $\gamma={(\theta_2-\theta_1)}/{2}.$ Therefore, a matrix $A\in \mathcal{M}_n$ is rotational sectorial if and only if $0\notin W(A).$ 
 
% We mention that when $A\in \mathcal{M}_n$ is rotational sectorial, i.e., $A\in \Pi_{\theta,\gamma}$ (not necessarily sectorial), then one can obtain the inequalities for the numerical radius $w(A)$ of the matrix $e^{i\theta}A$ (as $w(A)=w(e^{i\theta}A)$), by using similar arguments as used in this section to develop numerical radius inequalities of sectorial matrices.}

\section{\text{On special sectorial matrices}}

In this section, we obtain bounds for the numerical radius for a particular class of sectorial matrices. Let $A\in \mathcal{M}_n$ be a sectorial matrix with $\Im(A)<0.$ Then, it is clear from the Cartesian decomposition of $A$ that $iA$ is also a sectorial matrix. Moreover, $iA$ is accretive-dissipative. Conversely, if $A$ and $iA$ are both sectorial matrices, then $\Re(A)>0$ and $\Im(A)<0.$ 

Let  $A\in \mathcal{M}_n$ be such that 
\begin{eqnarray}\label{eq1}
W(A)\subseteq\{re^{-i\theta}~:~\theta_1\leq \theta\leq \theta_2\},
\end{eqnarray}
where $r>0$ and $\theta_1,\theta_2\in(0,\pi/2).$
Then, $A$ and $iA$ are both sectorial matrices with sectorial index $\theta_2$ and $\pi/2-\theta_1,$ respectively. Conversely, it is easy to observe that when $A$ and $iA$ are both sectorial then there exists $r>0$ and $\theta_1,\theta_2\in(0,\pi/2)$ such that $W(A)\subseteq\{re^{-i\theta}~:~\theta_1\leq \theta\leq \theta_2\}.$

\begin{theorem}\label{th4}
	Let $A\in \mathcal{M}_n$  satisfies the property \eqref{eq1}. Then 
	\begin{eqnarray*}
			\|A\| \leq \sqrt{1+cos^2\theta_1} \,w(A) \leq \sqrt{2}w(A).
	\end{eqnarray*}
\end{theorem}

\begin{proof}
	We only prove the first inequality, as second inequality follows trivially.
Since $iA$ is accretive-dissipative, and  $iA \in \Pi_{\pi/2-\theta_1},$ by  using Lemma \ref{lemma3} we have,
    \begin{eqnarray}\label{th4eq2}
	    \|A\| \leq \sqrt{1+cos^2\theta_1}w(A).
    \end{eqnarray}

\end{proof}

%\begin{example}\label{example1}
We note that when $A\in \mathcal{M}_n$ satisfies the property \eqref{eq1}, then there exists $\theta\in [0,\pi/2]$ such that $e^{i\theta}A\in \Pi_{(\theta_2-\theta_1)/2}$, and  by Lemma \ref{lemma3} we infer that
	\begin{eqnarray}\label{th4eq1}
		\|A\| \leq \sqrt{1+2sin^2(\theta_2-\theta_1)/2}\,w(A).
	\end{eqnarray}

%One can easily verify that the inequalities (\ref{th4eq2}) and (\ref{th4eq1}) are not comparable, in general.
%Therefore, combining the inequalities (\ref{th4eq2}) and (\ref{th4eq1}), we conclude that if $A\in \mathcal{M}_n$ satisfies the property \eqref{eq1}, then
%\begin{eqnarray}\label{th4eq3}
%	\|A\| \leq \min \left\{\sqrt{1+cos^2\theta_1},\sqrt{1+2sin^2(\theta_2-\theta_1)/2}\right\}w(A)\leq \sqrt{2}w(A).
%\end{eqnarray}

To prove the next theorem we need the following lemma, proved in \cite{kitt_LAA_2022}.

\begin{lemma}\label{lemma7} %$($\cite{kitt_LAA_2022}$)$
	Let $A,B \in \Pi_{\gamma}.$ Then 
	\begin{eqnarray*}
		w(AB) \leq (1+2sin^2\gamma) w(A)w(B).
	\end{eqnarray*} 
	Moreover, if $A,B \in \Pi_{\gamma}$ be accretive-dissipative, then
	\begin{eqnarray*}
		w(AB) \leq (1+sin^2\gamma) w(A)w(B).
	\end{eqnarray*}  
\end{lemma}

\begin{theorem}\label{th5}
Let $A,B\in \mathcal{M}_n$ which satisfy the property \eqref{eq1}. Then 
	\begin{eqnarray*}
		w(AB) \leq (1+cos^2\theta_1) w(A)w(B) \leq 2w(A)w(B).
	\end{eqnarray*}
\end{theorem}

\begin{proof}
	Clearly, $iA$ and $iB$ are accretive-dissipative, and  $iA,iB \in \Pi_{\pi/2-\theta_1}.$ Then  from Lemma \ref{lemma7} we get,
	\begin{eqnarray}\label{th5eq2}
			w(AB) \leq (1+cos^2\theta_1) w(A)w(B).
	\end{eqnarray}
 The second inequality follows easily.
\end{proof}

%\begin{example}\label{example2}
	We also note that when  $A,B\in \mathcal{M}_n$ satisfy the property \eqref{eq1}, then there exists $\theta\in [0,\pi/2]$ such that $e^{i\theta}A, e^{i\theta}B\in \Pi_{(\theta_2-\theta_1)/2}.$ Then by Lemma \ref{lemma7}, we have
	
	\begin{eqnarray}\label{th5eq1}
			w(AB) \leq (1+2sin^2(\theta_2-\theta_1)/2) w(A)w(B).
	\end{eqnarray}

%It is easy to observe that the inequalities (\ref{th5eq2}) and (\ref{th5eq1}) are not comparable.
%Therefore, from the inequalities (\ref{th5eq2}) and (\ref{th5eq1}) we conclude that 
%\begin{eqnarray}\label{th5eq3}
%	w(AB) \leq \min \left\{1+cos^2\theta_1,1+2sin^2(\theta_2-\theta_1)/2 \right\}w(A)w(B)\leq %2w(A)w(B).
%\end{eqnarray}

\vspace{.2cm}

Now, we observe that when $A\in \mathcal{M}_n$ satisfies the property (\ref{eq1}), then both $A,iA \in \Pi_{\gamma_1},$ where $\gamma_1=\max\{\theta_2,\pi/2-\theta_1\}.$ 
From this observation we prove the following theorem. First we need the following lemma,
 see  \cite[Th. 3.5]{kitt_LAA_2022}.

\begin{lemma}\label{lemma8} %$($\cite[Th. 3.5]{kitt_LAA_2022}$)$
	Let $A \in \Pi_{\gamma}\, ( \gamma \neq 0).$ Then 
	\begin{eqnarray*}
		w(A) \geq  \frac{csc\gamma}{2}\|A\|
		+ \frac{csc\gamma}{2}\left( \|\Im(A)\|-\|\Re(A)\|\right).
	\end{eqnarray*}  
\end{lemma}

\begin{theorem}\label{th6}
	Let $A\in \mathcal{M}_n$ be satisfies the property (\ref{eq1}). Then 
	\begin{eqnarray}\label{th6eq1}
		w(A) \geq  \frac{csc\gamma_1}{2}\|A\|
		+ \frac{csc\gamma_1}{2}\left| \|\Im(A)\|-\|\Re(A)\|\right|,
	\end{eqnarray}
	 where $\gamma_1=\max\{\theta_2,\pi/2-\theta_1\}.$
\end{theorem}

\begin{proof}
 Since $A$ and $iA$ are in $\Pi_{\gamma_1},$ by using Lemma \ref{lemma8} we get,
	\begin{eqnarray}\label{th6eq2}
		w(A) \geq  \frac{csc\gamma_1}{2}\|A\|
		+ \frac{csc\gamma_1}{2}\left( \|\Im(A)\|-\|\Re(A)\|\right)
	\end{eqnarray}
   and
    \begin{eqnarray}\label{th6eq3}
    	w(A) \geq  \frac{csc\gamma_1}{2}\|A\|
    	+ \frac{csc\gamma_1}{2}\left(\|\Re(A)\|-\|\Im(A)\|\right).
    \end{eqnarray}
Thus, combining (\ref{th6eq2}) and (\ref{th6eq3}) we get the desired inequality.
\end{proof}

\begin{remark}
	Since $\gamma_1 \in (0,\pi/2),$ $csc\gamma_1>1.$ Therefore, from Theorem \ref{th6} we get,
	\begin{eqnarray*}
		w(A) &\geq & \frac{csc\gamma_1}{2}\|A\|
		+ \frac{csc\gamma_1}{2}\left| \|\Im(A)\|-\|\Re(A)\|\right|\\
		&>& \frac{\|A\|}{2}
		+ \frac{\left| \|\Im(A)\|-\|\Re(A)\|\right|}{2}.
	\end{eqnarray*}
Thus, we would like to remark that 
Theorem \ref{th6} gives sharper bound than the existing bound \cite[Th. 2.1]{bhunia_LAA_2021},  namely,
\begin{eqnarray}\label{1p}
	w(A) &\geq & \frac{\|A\|}{2}
	+ \frac{\left| \|\Im(A)\|-\|\Re(A)\|\right|}{2}.
\end{eqnarray}
\end{remark}

Next refinement reads as:

\begin{theorem}\label{th7}
Let $A\in \mathcal{M}_n$  satisfies the property (\ref{eq1}). Then
	\begin{eqnarray}\label{th7eq1}
	    w^2(A) \geq  \frac{csc^2\gamma_1}{4}\|AA^*+A^*A\|
		+ \frac{csc^2\gamma_1}{2}\left| \|\Im(A)\|^2-\|\Re(A)\|^2\right|,
	\end{eqnarray}
	 where $\gamma_1=\max\{\theta_2,\pi/2-\theta_1\}.$
\end{theorem}

\begin{proof}
	Proceeding as Theorem \ref{th6} and using the inequality in Theorem \ref{th1} we obtain the desired inequality.
	
	 %Since, $A$ and $iA$ are in $\Pi_{\gamma_1},$ where $\gamma_1=\max\{\theta_2,\pi/2-\theta_1\}.$ Now, proceeding same way as Theorem \ref{th1} for $A$ and $iA,$ we get
%	\begin{eqnarray}\label{th7eq2}
%		w^2(A) \geq  \frac{csc^2\gamma_1}{4}\|AA^*+A^*A\|
%		+ \frac{csc^2\gamma_1}{2}\left( \|\Im(A)\|^2-\|\Re(A)\|^2\right)
%	\end{eqnarray}
%	and
%	\begin{eqnarray}\label{th7eq3}
%		w^2(A) \geq  \frac{csc^2\gamma_1}{4}\|AA^*+A^*A\|
%		+ \frac{csc^2\gamma_1}{2}\left(\|\Re(A)\|^2-\|\Im(A)\|^2\right),
%	\end{eqnarray}
%	respectively. Therefore, combining (\ref{th7eq2}) and (\ref{th7eq3}) we get the desired result.
\end{proof}

\begin{remark}
		Since $\gamma_1 \in (0,\pi/2),$ $csc^2\gamma_1>1.$ Hence, from Theorem \ref{th7} we get,
	\begin{eqnarray*}
	 w^2(A) &\geq&  \frac{csc^2\gamma_1}{4}\|AA^*+A^*A\|
	+ \frac{csc^2\gamma_1}{2}\left| \|\Im(A)\|^2-\|\Re(A)\|^2\right|\\
		&>& \frac{\|AA^*+A^*A\|}{4}
		+ \frac{\left| \|\Im(A)\|^2-\|\Re(A)\|^2\right|}{2}.
	\end{eqnarray*}
	Thus, we would like to remark that 
	Theorem \ref{th7} gives sharper bound than that in \cite[Th. 2.9]{bhunia_LAA_2021}, namely,
		\begin{eqnarray} \label{2p}
		w^2(A) &\geq& \frac{\|AA^*+A^*A\|}{4}
		+ \frac{\left| \|\Im(A)\|^2-\|\Re(A)\|^2\right|}{2}.
	\end{eqnarray}
\end{remark}

Finally we note that when $A\in \mathcal{M}_n$ with the numerical range $W(A)$ satisfying  the property
\begin{eqnarray}\label{eq2}
	W(A)\subseteq\{re^{i\theta}~:~\theta_1\leq \theta\leq \theta_2\},
\end{eqnarray}
where $r>0$ and $\theta_1,\theta_2\in(0,\pi/2),$ then $A$ and $-iA $ are both in $ \Pi_{\gamma_1}$, where $\gamma_1=\max \{ \theta_2, {\pi}/2-\theta_1 \}.$ Therefore, Theorems \ref{th6} and \ref{th7} also hold when $A\in \mathcal{M}_n$ satisfies  \eqref{eq2}.
  Further, we  see that Theorems \ref{th6} and \ref{th7} hold when $\theta_1, \theta_2 \in \{ 0, \pi/2\}$, see \eqref{1p} and \eqref{2p}.

\bibliographystyle{amsplain}
	
\end{document}